\documentclass{amsart}
\usepackage{mathrsfs,latexsym,amsfonts,amssymb}

\newtheorem{theorem}{Theorem}[section]
\newtheorem{lemma}[theorem]{Lemma}
\newtheorem{corollary}[theorem]{Corollary}
\newtheorem{question}[theorem]{Question}
\theoremstyle{definition}
\newtheorem{definition}[theorem]{Definition}
\newtheorem{proposition}[theorem]{Proposition}
\theoremstyle{remark}

\begin{document}

\title[Topologically subordered rectifiable spaces and compactifications]
{Topologically subordered rectifiable spaces and compactifications}

\author{Fucai Lin}
\address{(Fucai Lin): Department of Mathematics and Information Science,
Zhangzhou Normal University, Zhangzhou 363000, P. R. China}
\email{linfucai2008@yahoo.com.cn}

\thanks{Supported by the NSFC (No. 10971185) and the Educational Department of Fujian Province (No. JA09166) of China.}

%insert keywords
%insert subject class

%\date{\today}
\begin{abstract}
A topological space $G$ is said to be a {\it rectifiable space}
provided that there are a surjective homeomorphism $\varphi :G\times
G\rightarrow G\times G$ and an element $e\in G$ such that
$\pi_{1}\circ \varphi =\pi_{1}$ and for every $x\in G$ we have
$\varphi (x, x)=(x, e)$, where $\pi_{1}: G\times G\rightarrow G$ is
the projection to the first coordinate. In this paper, we mainly discuss the rectifiable spaces which are suborderable, and show that if a rectifiable space is suborderable then it is metrizable or a totally disconnected P-space, which improves a theorem of A.V. Arhangel'ski\v\i\ in \cite{A20092}. As an application, we discuss the remainders of the Hausdorff compactifications of GO-spaces which are rectifiable, and we mainly concerned with the following statement, and
under what condition $\Phi$ it is true.

\medskip
{\bf Statement} Suppose that $G$ is a non-locally compact
GO-space which is rectifiable, and that $Y=bG\setminus G$ has (locally) a
property-$\Phi$. Then $G$ and $bG$ are separable and metrizable.

Moreover, we also consieder some related matters about the remainders of the Hausdorff compactifications of rectifiable spaces.
\end{abstract}

\maketitle

\section{Introduction}
Recall that a {\it topological group} $G$ is a group $G$ with a
(Hausdorff) topology such that the product maps of $G \times G$ into
$G$ is jointly continuous and the inverse map of $G$ onto itself
associating $x^{-1}$ with arbitrary $x\in G$ is continuous. A {\it
paratopological group} $G$ is a group $G$ with a topology such that
the product maps of $G \times G$ into $G$ is jointly continuous. A
topological space $G$ is said to be a {\it rectifiable space}
provided that there are a surjective homeomorphism $\varphi :G\times
G\rightarrow G\times G$ and an element $e\in G$ such that
$\pi_{1}\circ \varphi =\pi_{1}$ and for every $x\in G$ we have
$\varphi (x, x)=(x, e)$, where $\pi_{1}: G\times G\rightarrow G$ is
the projection to the first coordinate. If $G$ is a rectifiable
space, then $\varphi$ is called a {\it rectification} on $G$. It is
well known that rectifiable spaces and paratopological groups are
all good generalizations of topological groups. In fact, for a
topological group with the neutral element $e$, then it is easy to
see that the map $\varphi (x, y)=(x, x^{-1}y)$ is a rectification on
$G$. However, there exists a paratopological group which is not a
rectifiable space; Sorgenfrey line (\cite[Example
1.2.2]{E}) is such an example. Also, the 7-dimensional sphere $S_{7}$ is
rectifiable but not a topological group \cite[$\S$ 3]{V1990}.
Further, it is easy to see that paratopological groups and
rectifiable spaces are all homogeneous. Recently, the study of rectifiable spaces has become an interesting topic in topological algebra, see \cite{A20092, A2009, G1996, LLL, LF}.

By a remainder of a Tychonoff space $X$ we understand the subspace
$bX\setminus X$ of a Hausdorff compactification $bX$ of $X$.
\bigskip

\section{Preliminaries}
Let ($X, \leq$) be a linearly ordered set. Then, a {\it linearly ordered topological
space} (abbreviated LOTS) is a triple ($X, \tau(\leq), \leq$ ), where $\tau(\leq)$ is the usual
order topology (i.e., open-interval topology) of the order $\leq$. Moreover, a space $X$ is a {\it generalized ordered space} (abbreviated GO-space) if $X$ is a subspace of a LOTS $\mathrm{Y}$ , where the order of $X$ is the one induced by the
order of $Y$, see \cite{LD}.

Recall that a space $X$ is {\it orderable} (resp. {\it suborderable}) if $X$ is homeomorphic to a
LOTS (resp. GO-space) \cite{NJ}. Thus, a space $X$ is orderable if and only if the topology of
$X$ coincides with the order topology by some linear order of $X$ \cite{VM}. The Sorgenfrey line \cite{E} is a suborderable space which is not orderable. It is well known that every compact or connected subspace of a suborderable space is orderable, see \cite{E}.

If it will cause no confusion, for a GO-space (or LOTS) ($X, \tau, \leq$ ), we shall omit
``$\tau$'' or $``\leq$''. Moreover, we shall sometimes use ``LOTS'' (resp. ``GO-spaces'')
instead of ``orderable spaces'' (resp. ``suborderable spaces'').

\begin{definition}\cite{A2008}
Let $\zeta$ and $\eta$ be any families of non-empty subsets of $X$.
\begin{enumerate}
\item The family $\zeta$ is called a {\it filterbase} on a space $X$ if, whenever $P_{1}$
and $P_{2}$ are in $\zeta$, there is a $P\in\zeta$ such that
$P\subset P_{1}\cap P_{2}$;

\item A filterbase $\zeta$ on a space $X$ is
said to {\it converge to a point} $x\in X$ if every open
neighborhood of $x$ contains an element of $\zeta$;

\item If $x\in X$
belongs to the closure of every element of a filterbase $\zeta$ on
$X$, we say that $\zeta$ {\it accumulates to} $x$ or {\it a cluster
point for} $\zeta$;

\item Two filterbases $\zeta$ and $\eta$ are called to
be {\it synchronous} if, for any $P\in\zeta$ and any $Q\in\eta$,
$P\cap Q\neq\emptyset$;

\item A space $X$ is called {\it bisequential}
 if, for every filterbase $\zeta$ on $X$ accumulating to
a point $x\in X$, there exists a countable filterbase $\xi$ on $X$
converging to the same point $x$ such that $\zeta$ and $\xi$ are
synchronous.
\end{enumerate}
\end{definition}

A filterbase $\mathscr{U}$ in a topological space $X$ is called an {\it open nest} \cite{A1990} if $\mathscr{U}$ consists of open subsets of $X$ and has the following property: For any $U, V\in\mathscr{U}$ either $U\subset V$ or $V\subset U.$ A space $X$ is {\it open $\pi$-nested} at a point $x\in X$ \cite{A1990} if there exists an open nest in $X$ converging to $x$, and $X$ is {\it open nested} at $x$ \cite{A1990} if there exists an open nest in $X$ which forms a local base for $X$ at $x$. Finally, a space is {\it open nested} if it is open nested at each of its points.

Recall that a family $\mathcal{U}$ of non-empty open sets of a space
$X$ is called a {\it $\pi$-base} if for each non-empty open set $V$
of $X$, there exists a $U\in\mathcal{U}$ such that $U\subset V$. The
{\it $\pi$-character} of $x$ in $X$ is defined by $\pi\chi(x,
X)=\mbox{min}\{|\mathcal{U}|:\mathcal{U}\ \mbox{is a local}\
\pi\mbox{-base at}\ x\ \mbox{in}\ X\}$. The {\it $\pi$-character of
$X$} is defined by $\pi\chi(X)=\mbox{sup}\{\pi\chi(x, X):x\in
X\}$. The
{\it character} of $x$ in $X$ is defined by $\chi(x,
X)=\mbox{min}\{|\mathcal{U}|:\mathcal{U}\ \mbox{is a local}\
\mbox{base at}\ x\ \mbox{in}\ X\}$. The {\it character of
$X$} is defined by $\chi(X)=\mbox{sup}\{\chi(x, X):x\in
X\}$.

A space $X$ is {\it countably compact} if for each countably open cover of $X$ has a finite subcover. A space $X$ is {\it locally compact} if every point of $G$ has a compact neighborhood. A Tychonoff space $X$ is {\it pseudocompact} if every real-valued continuous function on $X$ is bounded. A space $X$ is {\it $P$-space} if each $G_{\delta}$-set of $X$ is open in $X$.

\begin{theorem}\cite{C1992, G1996, V1989}\label{t9}
A topological space $G$ is rectifiable if and only if there exists $e\in G$ and two
continuous maps $p: G^{2}\rightarrow G$, $q: G^{2}\rightarrow G$
such that for any $x\in G, y\in G$ the next
identities hold:
$$p(x, q(x, y))=q(x, p(x, y))=y\ \mbox{and}\ q(x, x)=e.$$
\end{theorem}

Given a rectification $\varphi$ of the space $G$, we may obtain the mappings $p$ and $q$ in Theorem~\ref{t9} as follows. Let $p=\pi_{2}\circ\varphi^{-1}$ and
$q=\pi_{2}\circ\varphi$. Then the mappings $p$ and $q$ satisfy the identities in
Theorem~\ref{t9}, and both are open mappings.

Let $G$ be a rectifiable space, and let $p$ be the multiplication on
$G$. Further, we sometimes write $x\cdot y$ instead of $p(x, y)$ and
$A\cdot B$ instead of $p(A, B)$ for any $A, B\subset G$. Therefore,
$q(x, y)$ is an element such that $x\cdot q(x, y)=y$; since $x\cdot
e=x\cdot q(x, x)=x$ and $x\cdot q(x, e)=e$, it follows that $e$ is a right neutral
element for $G$ and $q(x, e)$ is a right inverse for $x$. Hence a
rectifiable space $G$ is a topological algebraic system with
operations $p$ and $q$, a 0-ary operation $e$, and identities as above. It is
easy to see that this algebraic system need not to satisfy the
associative law about the multiplication operation $p$. Clearly,
every topological loop is rectifiable.

In section 3, we mainly discuss the rectifiable spaces which are suborderable, and show that if a rectifiable space is suborderable then it is metrizable or a totally disconnected P-space, which generalizes a theorem of A.V. Arhangel'ski\v\i\ in \cite{A20092}. Moreover, we discuss the orderability of rectifiable spaces which are suborderable. In section 4, we mainly discuss the remainders of the rectifiable spaces, and in particular, we discuss the remainders of the rectifiable spaces which are suborderable.

All spaces are $T_1$ and regular unless stated otherwise.
The notation $\mathbb{N}$ denotes the set of all positive natural numbers. The letter $e$
denotes the neutral element of a group and the right neutral element
of a rectifiable space, respectively. Readers may refer to
\cite{A2008, E, Gr1984} for notations and terminology not
explicitly given here.

\bigskip
\section{suborderable rectifiable spaces}
In \cite{A20092}, A.V. Arhangel'ski\v\i\ proved the following theorem.

\begin{theorem}\cite{A20092}\label{t4}
If $G$ is an orderable rectifiable space, then $G$ is metrizable or P-space.
\end{theorem}

In this section, we shall generalize this theorem, see Theorem~\ref{t2}. Firstly, we give some lemmas.

\begin{lemma}\label{l0}
If a rectifiable space $G$ is open $\pi$-nested at some point $a\in G$ then $G$ is open nested.
\end{lemma}

\begin{proof}
Since $G$ is homogeneous, we may assume that $a$ is the right neutral element $e$ of the rectifiable space $G$. Take an open nest $\mathscr{A}$ converges to $e$. Put $\phi=\{q(U, U): U\in\mathscr{A}\}$. Since the mappings $\pi_{2}$ and $\varphi$ are open, the mapping $q$ is open. Moreover, it is easy to see that, for each $x\in G$, $q(x, x)=\pi_{2}\circ\phi(x, x)=\pi_{2}(x, e)=e$. Therefore, for each $U\in\mathscr{A}$, $q(U, U)$ is an open neighborhood of $e$.  Then $\phi$ is an open nest. Next, we only prove that $\phi$ converges to $e$. Indeed, let $V$ be an open neighborhood of $e$ in $G$. Since $q$ is continuous at point $(e, e)$, there exists open neighborhood $W$ of $e$ such that $q(W, W)\subset V$. Then there is an $U\in\mathscr{A}$ such that $U\subset W$ since $\mathscr{A}$ converges to $e$. Thus we have $q(U, U)\subset q(W, W)\subset V$, which  implies that $\phi$ is a base for $G$ at the point $e$. By the homogeneity of the space $G$ we have $G$ is open nested.
\end{proof}

A space is {\it biradial} \cite{A1990} if for each filterbase $\mathscr{F}$ in $X$ accumulating to a point $x\in X$, there exists a chain  $\mathscr{C}$ (i.e., a collection linearly ordered by $\supseteq$) in $X$ converging to $x$ and synchronous with $\mathscr{F}$.

\begin{lemma}\label{l1}\cite{A1990}
Every biradial space is open $\pi$-nested at every point.
\end{lemma}

\begin{lemma}\label{l2}\cite{A1990}
Let $X$ be an open nested space and $\chi(x, X)=\kappa$ for all $x\in X$. Then for any family $\gamma$ of open sets in $X$ such that $|\gamma|<\kappa$ the set $\bigcap\gamma$ is open.
\end{lemma}

\begin{lemma}\label{l4}\cite{LF}
Every bisequential rectifiable space is metrizable.
\end{lemma}

Since every GO-space is biradial, we have the following theorem by Lemmas~\ref{l0}, \ref{l1} and~\ref{l4}.

\begin{theorem}\label{t0}
If GO-space $G$ is rectifiable, then $G$ is metrizable or P-space.
\end{theorem}

\begin{proof}
Since every GO-space is biradial, $G$ is first-countable or P-space by Lemmas~\ref{l0}, \ref{l1} and~\ref{l2}. Therefore, $G$ is metrizable or P-space by Lemma~\ref{l4}.
\end{proof}

\begin{corollary}\label{c0}
If a GO-space $G$ which is a rectifiable space, and also, is locally countably compact or is of point-countable type\footnote{Recall that a space $X$ is of {\it point-countable type} \cite{E} if every
point $x\in X$ is contained in a compact subspace of
$X$ with a countable outer base of open neighborhoods in $X$.}, then $G$ is metrizable.
\end{corollary}

Let $G$ be a rectifiable space with operations $p$ and $q$ as defined
after Theorem~\ref{t9}. A subset $A\subset G$ is said to be $p, q$ {\it stable} if
$p(A, A)\subset A$ and $q(A, A)\subset A$.

\begin{lemma}\label{l3}
Let $G$ be a GO-space which is a rectifiable space. If $C$ is the connected component contains the right neutral element of $G$ and if $C$ has at least two points, then $C$ is open and $p, q$ stable in $G$.
\end{lemma}

\begin{proof}
Since $G$ is rectifiable, the subspace $C$ is homogeneous, and hence $C$ is open in $G$ by \cite[Lemma 2.2]{LC}. Also, since $C$ is is the connected component contains the right neutral element of $G$, it easy to check that $p(C, C)\subset C$ and $q(C, C)\subset C$. Therefore, $C$ is open and $p, q$ stable.
\end{proof}

In \cite{A20092}, A.V. Arhangel'ski\v\i\  proved that a connected and orderable rectifiable space is homeomorphic to the space $\mathbb{R}$ of the real numbers. He made a small mistake. His result is `` a
connected orderable space $G$ is rectifiable and has more than one point, then $G$ is
homeomorphic to the space $\mathbb{R}$ of the real numbers''. Therefore, we have the following lemma.

\begin{lemma}\label{l6}
If a connected GO-space $G$ is rectifiable and has more than one point, then $G$ is homeomorphic to the space $\mathbb{R}$ of the real numbers.
\end{lemma}

\begin{theorem}\label{t2}
If a GO-space $G$ is rectifiable, then $G$ is metrizable or a totally disconnected\footnote{A space is {\it totally disconnected} if each component is one point.} P-space.
\end{theorem}

\begin{proof}
Suppose that $G$ is non-metrizable. Then it follows from Theorem~\ref{t0} that $G$ is a P-space. Let $G$ be not totally disconnected. Then the connected component $C$ containing the identity of $G$ is open and $p, q$ stable by Lemma~\ref{l3}. Then $C$ is homeomorphic to the space $\mathbb{R}$ of the real numbers by Lemma~\ref{l6}. Then $G=\cup\{x\cdot C: x\in G\}$ since $e\in C$. Since $C$ is connected, $G$ is the topological sum of real lines. Moreover, since $\{x\cdot C: x\in G\}$ is a disjoint family since $C$ is connected. Therefore, $G$ is the topological sum of real lines. Then $G$ is metrizable, which is a contradiction. Hence $G$ is totally disconnected.
\end{proof}

\begin{theorem}\label{t1}
Let a GO-space $G$ be metrizable. If $G$ is a rectifiable space, then $G$ is orderable.
\end{theorem}

\begin{proof}
If $G$ is totally disconnected, then $G$ is orderable by \cite[Lemma 8]{TY}. Let $G$ be not totally disconnected. Then $G$ is the topological sum of real lines by the proof of Theorem~\ref{t2}, and therefore, $G$ is orderable \cite{HH}.
\end{proof}

In \cite{LC}, C. Liu, M. Sakai and Y. Tanaka proved the following theorem:

\begin{theorem}\cite{LC}
Let $G$ be a topological group. Then the following (1) and (2) are equivalent. When $G$ is non-metrizable, the following (1), (2) and (3) are equivalent.
\begin{enumerate}
\item $G$ is orderable;

\item $G$ is suborderable;

\item $G$ is a biradial space.
\end{enumerate}
\end{theorem}

Therefore, we have the following question:

\begin{question}
If a non-metrizable GO-space $G$ which is a rectifiable space, is $G$ orderable?
\end{question}

By Theorem~\ref{t2}, we have the following theorem.

\begin{theorem}\label{t10}
Let $G$ be a rectifiable space which is a GO-space. Then, $G$ is metrizable if one of the following properties (1)-(7) holds:
\begin{enumerate}
\item $G$ is locally separable;

\item all points of $G$ are $G_{\delta}$-sets;

\item $G$ is a quasi-$k$-space\footnote{A space is a {\it quasi-$k$-space}
\cite{NJ1} if it is determined by a cover of countably compact subsets.};

\item $t(G)\leq\omega$;

\item $G$ is (locally) a $\Sigma$-space\footnote{A space $X$ is a $\Sigma$-space, if
there exist a $\sigma$-locally finite closed cover $\mathscr{F}$ in $X$ , and a cover $\mathscr{C}$ of countably
compact closed subsets in $X$ such that, for $C\subset U$ with $C\in\mathscr{C}$ and $U$ open
in $X,$ $C\subset F\subset U$ for some $F\in\mathscr{F}$.} ;

\item some countably compact subset of $G$ is infinite;

\item some connected subset of $G$ contains at least two points;
\end{enumerate}
\end{theorem}

\begin{proof}
Let $G$ be non-metrizable. Then $G$ is totally disconnected $P$-space by Theorem~\ref{t2}, which is obvious contradictory to (7). Also, it easily seen that every countably compact subset of $G$ is finite which is a contradiction with (6). Moreover, if $G$ satisfies one of properties (1)-(5), then it is easy to see that $G$ is discrete space by \cite[Proposition 1.3]{TY1}, which is a contradiction. Therefore, $G$ must be metrizable if $G$ satisfies one of the properties (1)-(7).
\end{proof}

\begin{theorem}
Let $G$ be a rectifiable space. Then $G$ is metrizable if and only if $G^{\omega}$ is biradial.
\end{theorem}

\begin{proof}
In \cite[Corollary 2.7]{LC}, C. Liu, M. Sakai and Y. Tanaka proved that $G^{\omega}$ is biradial if and only if $G$ is bisequential. Since a bisequentially rectifiable space is metrizable by Lemma~\ref{l4},  it follows that $G^{\omega}$ is biradial if and only if $G$ is metrizable.
\end{proof}

\begin{theorem}
Let $(G, \tau)$ be any rectifiable space which  is not totally disconnected. Then $(G, \tau)$ is orderable as a topological space if and only if $(G, \tau)$ contains an open connected and $p, q$ stable subspace $H$ topologically isomorphic with $\mathbb{R}$.
\end{theorem}

\begin{proof}
Necessity. Suppose that $(G, \tau)$ is orderable as a topological space. Then the connected component $C$ containing the identity of $G$ is open and $p, q$ stable by Lemma~\ref{l3}. Then $C$ is homeomorphic to the space $\mathbb{R}$ of the real numbers by Lemma~\ref{l6}.

Sufficiency. It is well known that a non-trivial connect space is topological sum of real lines if and only if it is locally separable, metrizable, suborderable. Since $H$ is an open connected subspace topologically isomorphic with $\mathbb{R}$, $G$ is the topological sum of real lines, and hence $(G, \tau)$ is orderable \cite[Theorem 9]{HH}.
\end{proof}

Let $(X, \leq)$ be a linearly ordered set. A point $x\in X$ is {\it isolated from below} if it is not the supremum of the points strictly below it. Let $\alpha$ be a cardinal number. If $x$ is not isolated from below, then $x$ has {\it cofinality $\alpha$ from below} if $x$ is the supremum of the set of $\alpha$ points strictly below it, and if $\alpha$ is the smallest cardinal with this property.

The concepts of ``{\it isolated from above}'' and ``{\it cofinality from above}''
 are defined dually.

\begin{lemma}\label{l5}
Let a rectifiable space $G$ be orderable. If the right neutral
element $e$ of $G$ is isolated neither from above nor from below, then its
cofinality from above equals its cofinality from below.
\end{lemma}

\begin{proof}
Let the cofinality of $e$ from below be $\alpha$, while its cofinality
from above is strictly greater. Suppose that $\{U_{\beta}: \beta<\alpha\}$ is a system of neighborhoods
of $e$ satisfying the following conditions (1)-(3):

(1) $U_{\beta}=(a_{\beta}, b_{\beta})$ for some points $a_{\beta}, b_{\beta}\in G;$

(2) $q(U_{\beta+1}, U_{\beta+1})\subset U_{\beta}$ for each $\beta<\alpha$;

(3) $\lim a_{\beta}=e$.

Let $U=\bigcap_{\beta<\alpha} U_{\beta}$. Obviously, $U$ contains an interval $[e, a]$ with $e\neq a$. Choose $b\in (e, a)$. For each $\beta<\alpha$, since $a_{\beta}\not\in U_{\beta}$, we have $b\cdot a_{\beta}\not\in U_{\beta+1}$ (if not, suppose that $b\cdot a_{\beta}\in U_{\beta+1}$, and hence $a_{\beta}=q(b, b\cdot a_{\beta})\in q(U_{\beta+1}, U_{\beta+1})\subset U_{\beta}$ by (2), which is a contradiction.). Thus $b\cdot a_{\beta}\not\in [e, a]$ for each $\beta<\alpha$, which implies that $b$ is not in the closure of $\{b\cdot a_{\beta}: \beta<\alpha\}$. However, it follows from (3) that $b\cdot a_{\beta}\rightarrow b\cdot e=b$, which is a contradiction.
\end{proof}

\begin{theorem}\label{t3}
Let a rectifiable space $G$ be orderable. Then there is a
totally ordered base for the neighborhoods of the right neutral element of $G$.
\end{theorem}

\begin{proof}
If $e$ is isolated from above or below, then the result is obvious. If $e$ is isolated neither from above nor below, then it is easy to see that the result is true by
Lemma~\ref{l5}.
\end{proof}

\begin{lemma}\label{l7}
Let $G$ be a rectifiable space in which the right neutral element has a totally ordered local base. Then $G$ satisfies one of the following conditions (1)-(2):
\begin{enumerate}
\item $G$ is metrizable;

\item $G$ has a totally ordered local base at the right neutral consisting of clopen and $p, q$ stable subspaces.
\end{enumerate}
\end{lemma}

\begin{proof}
Let $G$ be non-metrizable. Suppose that $\{U_{\beta}: \beta<\alpha\}$ is a local base
at $e$, where $\alpha>\omega$. Fix an arbitrary $U_{\beta}$. Then there exists a countable family $\{U_{\beta n}: n\in\mathbb{N}\}$ of neighborhoods at $e$ such that $U_{\beta (n+1)}\cdot U_{\beta (n+1)}\subset U_{\beta n}\subset U_{\beta}$ and $q(U_{\beta (n+1)}, U_{\beta (n+1)})\subset U_{\beta n}$ for each $n\in\mathbb{N}$. Indeed, let $U_{\beta 1}=U_{\beta}$. Since $U_{\beta 1}$ is an open neighborhood of $e$, there exists an open neighborhood $U_{\beta 2}$ of $e$ such that $U_{\beta 2}\cdot U_{\beta 2}\subset U_{\beta 1}$ and $q(U_{\beta 2}, U_{\beta 2})\subset U_{\beta 1}$. Now suppose for each $k\leq n$ has been defined well. For $k=n+1$, since $U_{\beta n}$ is an open neighborhood of $e$, there exists an open neighborhood $U_{\beta (n+1)}$ of $e$ such that $U_{\beta (n+1)}\cdot U_{\beta (n+1)}\subset U_{\beta n}$ and $q(U_{\beta (n+1)}, U_{\beta (n+1)})\subset U_{\beta n}$. By induction we have defined the countable family $\{U_{\beta n}: n\in\mathbb{N}\}$ of neighborhoods at $e$.
Let $B_{\beta}=\bigcap_{n\in\mathbb{N}} U_{\beta n}$. Then $B_{\beta}$ is open and $p, q$ stable of $G$, and hence it is closed. By induction, it is easy to see that we can arrange a family of clopen and $p, q$ stable subspaces $\{B_{\gamma}: \gamma\in\alpha\}$ which is totally ordered local base.
\end{proof}

By Theorem~\ref{t3} and Lemma~\ref{l7}, we have the following theorem.

\begin{theorem}\label{t5}
Let a rectifiable space $G$ be orderable. Then $G$ satisfies one of the following conditions (1)-(2):
\begin{enumerate}
\item $G$ is metrizable;

\item $G$ has a totally ordered local base at the right neutral consisting of clopen and $p, q$ stable subspaces of $G$.
\end{enumerate}
\end{theorem}

{\bf Remark} From Theorem~\ref{t5} it follows that any non-metrizable, orderable rectifiable space is a $P$-space, and thus we obtain Theorem~\ref{t4} again.

However, the following question is still open.

\begin{question}
Let $G$ be a compact, totally disconnected rectifiable space. Does there exist a cardinal number $\tau$ such that $G$, regard as a topological space, is homeomorphic with the space $\{0, 1\}^{\tau}$?  (that is, is every compact, totally disconnected rectifiable space a topological group?)
\end{question}

\bigskip

\section{compactification of rectifiable spaces}
In this section, we assume that all spaces are Tychonoff. Firstly, we give some lemmas.

\begin{lemma}({\bf Henriksen and Isbell} \cite{H1958})\label{l12}
 A space
$X$ is of countable type\footnote{Recall that a space $X$ is of {\it countable type} \cite{E} if every
compact subspace $F$ of $X$ is contained in a compact subspace
$K\subset X$ with a countable outer base of open neighborhoods in $X$.} if and only if its remainder in any (in
some) compactification of $X$ is Lindel\"{o}f.
\end{lemma}

\begin{lemma}\label{l14}\cite{A3}
If $G$ is a topological group, and some remainder of $G$ is not
pseudocompact, then every remainder of $G$ is Lindel\"{o}f.
\end{lemma}

\begin{lemma}\label{l13}
Let a GO-space $G$ be rectifiable, and $Y$ be the remainder of some Hausdorff compactification of $G$. If $Y$ is Lindel\"{o}f, then $G$ is metrizable.
\end{lemma}

\begin{proof}
Since $Y$ is Lindel\"{o}f, it follows from Lemma~\ref{l12} that $G$ is of countable type.

Case 1: There exists some countably compact subset $F$ of $G$ that is infinite.

Then $G$ is metrizable by Theorem~\ref{t10}.

Case 2: There is no countably compact subset of $G$ that is infinite.

Since $G$ is of countable type, there exists a compact subset $F$ such that the right neutral element $e\in F$ and $F$ has a countable neighborhood base. The compact subset $F$ is finite since there does not exist any compact subset of $G$ that is infinite. Therefore, the point $e$ has a countable neighborhood base, and thus $G$ is metrizable by Lemma~\ref{l4}.
\end{proof}

\begin{theorem}\label{t11}
Let a GO-space $G$ be rectifiable, and $Y$ be the remainder of some Hausdorff compactification of $G$. If $Y$ is of countable pseudocharacter, then one of the following conditions holds:
\begin{enumerate}
\item $G$ is metrizable;

\item $Y$ is countably compact.
\end{enumerate}
\end{theorem}

\begin{proof}
It follows from Lemma~\ref{l14} that $G$ is Lindel\"{o}f or $Y$ is pseudocompact.

By Lemma~\ref{l13}, we may assume that $Y$ is pseudocompact, and therefore, $Y$ is first-countable. Indeed, fix an arbitrary point $y_{0}\in Y$. Since $Y$ is a Tychonoff space with a countable pseudocharacter, it is easy to see that there exists a sequence $\eta=\{U_{n}: n\in\mathbb{N}\}$ of open neighborhoods at $y_{0}$ in $Y$ such that $\bigcap_{n\in\mathbb{N}}\overline{U_{n}}^{Y}=\{y_{0}\}$ and $\overline{U_{n+1}}^{Y}\subset U_{n}$ for each $n\in \mathbb{N}$, where each $\overline{U_{n}}^{Y}$ denotes the closure of $U_{n}$ in $Y$. Then $\eta$ is a local base at $y_{0}$ in $Y$. Suppose that $\eta$ is not a local base at $y_{0}$ in $Y$. Then there exists an open neighborhood $U$ at $y_{0}$ in $Y$ such that $U_{n}\setminus \overline{U}^{Y}\neq\emptyset$ for each $n\in \mathbb{N}$. Then $\{U_{n}\setminus\overline{U}: n\in\mathbb{N}\}$ is a decreasing sequence of non-empty open subsets of $Y$, and hence it follows from \cite[Theorem 3.10.23]{E} that $\bigcap_{n\in\mathbb{N}}\overline{U_{n}\setminus\overline{U}^{Y}}^{Y}\neq\emptyset$. However, $y_{0}\not\in\bigcap_{n\in\mathbb{N}}\overline{U_{n}\setminus\overline{U}^{Y}}^{Y}\subset\bigcap_{n\in\mathbb{N}}\overline{U_{n}}=\{y_{0}\}$, and hence $\bigcap_{n\in\mathbb{N}}\overline{U_{n}\setminus\overline{U}^{Y}}^{Y}=\emptyset$, which is a contradiction.

We assume that $Y$ is non-countably compact. It follows, by a standard argument,
that $G$ has a countable $\pi$-base at some point which is an accumulation point of some countable subset of $Y$. Therefore, $G$ is metrizable.
\end{proof}

{\bf Note.} Since Sorgenfrey line is first-countable and non-metrizable, it is suborderable space which is not a rectifiable space. Moreover, it is well known that Sorgenfrey line is not countably compact. However, there exists a Hausdorff compactification of the Sorgenfrey line such that its remainder is homeomorphic to the Sorgenfrey line, such as the two-arrows space of P.S. Alexandroff and P.S. Urysohn \cite[Exercise 3.10.C]{E}.

By the proof of Theorem~\ref{t11}, it is easy to see the following proposition holds since a rectifiable space is metrizable if it is of countable $\pi$-character \cite{G1996}.

\begin{proposition}
Let a non-locally compact GO-space $G$ be rectifiable, and $Y$ be the remainder of some Hausdorff compactification of $G$. If $Y$ is of countable $\pi$-character and non-countably compact, then $G$ is metrizable.
\end{proposition}

Recall that the space $X$ has a {\it quasi-$G_\delta$-diagonal} provided there is a sequence $\{\mathcal{G}(n): n\in \mathbb{N}\}$ of collections of open sets with property that, given distinct points $x, y\in X$, there is some $n$ with $x\in st(x, \mathcal{G}(n))\subset X\setminus \{y\}$. If each $\mathcal{G}(n)$ covers $X$, then $X$ has a {\it $G_{\delta}$-diagonal}. Obviously, ``$X$ has a $G_\delta$-diagonal" implies ``$X$ has a quasi-$G_\delta$-diagonal".

\begin{lemma}\label{l11}\cite{LLL}
Let $G$ be a non-locally compact, paracompact rectifiable space, and
$Y=bG\setminus G$ have locally quasi-$G_\delta$-diagonal. Then $G$ and $bG$ are
separable and metrizable.
\end{lemma}

\begin{lemma}\cite[Proposition 4.2]{BH1}\label{l78}
If $G$ is a GO-space which is a rectifiable space, then $G$ is hereditarily paracompact.
\end{lemma}

Therefore, it follows from Lemmas~\ref{l11} and~\ref{l78} that the following theorem is obtained.

\begin{theorem}\label{t15}
Let a non-locally compact GO-space $G$ be rectifiable, and $Y$ be the remainder of some Hausdorff compactification of $G$. If $Y$ has locally a quasi-$G_{\delta}$-diagonal, then $G$, $Y$ and $bG$ are separable and metrizable.
\end{theorem}

However, the following question is still open.

\begin{question}
Let a GO-space $G$ be rectifiable, and $Y$ be the remainder of some Hausdorff compactification of $G$. If $Y$ is first-countable, is $G$ metrizable?
\end{question}

\begin{theorem}\label{t14}
Let a GO-space $G$ be rectifiable, and $Y$ be the remainder of some Hausdorff compactification of $G$. If $Y$ is Ohio-complete\footnote{We call a space $X$ is {\it Ohio complete} \cite{A} if in each
compactification $bX$ of $X$ there is a $G_{\delta}$-subset $Z$ such
that $X\subset Z$ and each point $y\in Z\setminus X$ is separated
from $X$ by a $G_{\delta}$-subset of $Z$.}, then $G$ is metrizable.
\end{theorem}

\begin{proof}
Since $Y$ is Ohio-complete, the space $G$ is $\sigma$-compact or $G$ is of countable type.

Case 1: There exists some countably compact subset $F$ of $G$ that is infinite.

Then $G$ is metrizable by Theorem~\ref{t10}.

Case 2: No countably compact subset of $G$ is infinite.

If $G$ is $\sigma$-compact, $G$ is finite, and hence $G$ is discrete; if $G$ is of countable type, then
$G$ is metrizable by the proof of Theorem~\ref{t11}.
\end{proof}

In \cite{A}, we know that each $p$-space is Ohio-complete, and therefore, we have the following corollary. The $p$-spaces are a class of generalized metric spaces \cite{A4}. It is
well-known that every metrizable space is a $p$-space, and every
$p$-space is of countable type.

\begin{corollary}\label{c1}
Let a GO-space $G$ be rectifiable, and $Y$ be the remainder of some Hausdorff compactification of $G$. If $Y$ is a $p$-space, then $G$ is separable and metrizable.
\end{corollary}

\begin{proof}
It is easy to see by Lemma~\ref{l12} and Theorem~\ref{t14}.
\end{proof}

Next, we consider the remainders with locally a point-countable $k$-network\footnote{ Let $\mathscr{P}$ be a family of subsets of
a space $X$. The family $\mathscr{P}$ is called a {\it $k$-network} \cite{PO} if
whenever $K$ is a compact subset of $X$ and $K\subset U\in \tau
(X)$, there is a finite subfamily $\mathscr{P}^{\prime}\subset
\mathscr{P}$ such that $K\subset \cup\mathscr{P}^{\prime}\subset U$.} of GO-spaces which are rectifiable. First, we give some technical lemmas.

\begin{lemma}\cite{A20091}\label{l18}
Suppose that $B=X\cup Y$, where $B$ is a compact space, and $X, Y$ are dense nowhere compact subspaces of $B$. Suppose also that each compact subset of $Y$ is contained in a compact $G_{\delta}$-subset of $Y$. Then each locally finite (in $X$) family of non-empty open subsets of $X$ is countable.
\end{lemma}

\begin{lemma}\label{l15}
Let $G$ be a GO-space which is rectifiable. Then $G$ is
Lindel$\ddot{\mbox{o}}$f if and only if there exists a
compactification $bG$ such that for any compact subset $F\subset
Y=bG\setminus G$ is contained in a $G_{\delta}$-subset of $Y$.
\end{lemma}

\begin{proof}
If $G$ is Lindel$\ddot{\mbox{o}}$f, then $Y$ is of countable type by
Lemma~\ref{l12}. Therefore, we only need to show the
sufficiency. It follows from Lemma~\ref{l14} that $Y$ is
pseudocompact or Lindel$\ddot{\mbox{o}}$f.

Case 1: The space $Y$ is Lindel$\ddot{\mbox{o}}$f.

It follows from Lemma~\ref{l13} that $G$ is metrizable, and
hence $G$ is a paracompact space. Therefore, the space $G$ is
Lindel$\ddot{\mbox{o}}$f by Lemma~\ref{l18}.

Case 2: The space $Y$ is pseudocompact.

For each compact subset $F$ of $Y$, there exists a compact subset $L$ of $Y$ such that $F\subset L$ and $L$ is a $G_{\delta}$-set in $Y$. Since $Y$ is
pseudocompact, it is well known that every compact subset $L$ has a countably open neighborhood base, and hence $Y$ is of
countable type. Therefore $G$ is Lindel$\ddot{\mbox{o}}$f by Lemma~\ref{l12}.
\end{proof}

\begin{theorem}\label{t12}
Let $G$ be a non-locally compact GO-space which is rectifiable, and $Y=bG\setminus G$ a remainder of $G$.
Then the following conditions are equivalent:
\begin{enumerate}
\item the space $Y$ is of subcountable type\footnote{Recall that a space $X$ is of {\it subcountable type} \cite{A20091} if every
compact subspace $F$ of $X$ is contained in a compact $G_{\delta}$
subspace $K$ of $X$.};

\item the space $Y$ is of countable type.
\end{enumerate}
\end{theorem}

\begin{proof}
It is easy to see by Lemma~\ref{l15} and Lemma~\ref{l12}.
\end{proof}

\begin{theorem}\label{t13}
Let $G$ be a non-locally compact GO-space which is rectifiable, and $Y=bG\setminus G$ a remainder of $G$.
If $Y$ is $\kappa$-perfect\footnote{Recall that a space $X$ is of
{\it $\kappa$-perfect} \cite{A20091} if every compact subspace $F$ is
a $G_{\delta}$ subspace of $X$.}, then $Y$ is first countable.
\end{theorem}

\begin{proof}
It follows from Theorem~\ref{t12} that $Y$ is of countable type.
Since every point of $Y$ is a $G_{\delta}$-point, the space $Y$ is first
countable.
\end{proof}

Let $\mathcal{A}$ be a collection of subsets of $X$. The collection $\mathcal{A}$ is a {\it
$p$-metabase} \cite{BD} for $X$ if for distinct points $x, y\in X$, there
exists an $\mathcal{F}\in \mathcal{A}^{<\omega}$ such that $x\in
(\cup\mathcal{F})^{\circ}\subset\cup\mathcal{F}\subset X-\{y\}$.

The following lemma is an easy exercise.

\begin{lemma} \label{l19}
Suppose that $X$ is a Lindel\"{o}f space with locally a
point-countable $p$-metabase. Then $X$ has a point-countable
$p$-metabase.
\end{lemma}

\begin{lemma}\label{l17}\cite{LFC2009}
Suppose that $X$ has a point-countable $p$-metabase.
Then each countably compact subset of $X$ is a compact, metrizable,
$G_{\delta}$-subset of $X$.
\end{lemma}

\begin{lemma}\cite{A} \label{l20}
If $X$ is a Lindel\"{o}f $p$-space, then any remainder of $X$
is a Lindel\"{o}f $p$-space.
\end{lemma}

\begin{lemma}\label{l16}\cite{LLL}
Let $G$ be a non-locally compact rectifiable space, and
$Y=bG\setminus G$ have locally a point-countable base. Then $G$ and $bG$ are
separable and metrizable.
\end{lemma}

\begin{theorem}
Let a non-locally compact GO-space $G$ be rectifiable, and $Y=bG\setminus G$. If $Y$ has a locally point-countable $p$-metabase, then $G$ and $bG$ are separable and metrizable.
\end{theorem}

\begin{proof}
It follows from Lemma~\ref{l14} that $Y$ is Lindel\"{o}f or pseudocompact.

Case 1: $Y$ is Lindel\"{o}f.

Then $G$ is metrizable by Lemma~\ref{l13}, and hence $G$ is paracompact. Since $G$ has a locally point-countable $p$-metabase, it follows from Lemma~\ref{l17} that each compact subset of $Y$ is a $G_{\delta}$-subset of $Y$, and hence it follows from Lemma~\ref{l18} that $G$ is Lindel\"{o}f. Therefore, $G$ is separable and metrizable, and hence  $Y$ is a Lindel\"{o}f $p$-space by Lemma~\ref{l20}. Since $Y$ has a point-countable $p$-metabase by Lemma~\ref{l19}, the space $Y$ is separable and metrizable by \cite{GMT1}. Then $G$ and $bG$ are separable and metrizable by Theorem~\ref{t15}.

Case 2. $Y$ is pseudocompact.

By the proof of \cite[Theorem 6.21]{LF}, we see that $Y$ is $\kappa$-perfect, and hence $Y$ is first-countable by Theorem~\ref{t13}. Since $Y$ has a locally point-countable $p$-metabase, $Y$ has locally a point-countable base by \cite[Corollary 2.1.12]{Ls3}. It follows from Lemma~\ref{l16} that $G$ and $bG$ are
separable and metrizable.
\end{proof}

However, the following is still open.

\begin{question}
Let be a non-locally compact rectifiable space, and $Y=bG\setminus G$. If $Y$ has a locally point-countable $p$-metabase, are $G$ and $bG$ separable and metrizable?
\end{question}

Finally, we shall discuss some related problems about the remainders of Hausdorff compactificatins of rectifiable spaces.

\begin{lemma}\label{l8}\cite{Gr1984}
If a Tychonoff countably compact space $X$ is the union of a countable family of $D$-spaces\footnote{A {\it neighborhood assignment} for a space $X$ is a function
$\varphi$ from $X$ to the topology of $X$ such that $x\in \varphi
(x)$ for each point $x\in X$. A space $X$ is a {\it
D-space} \cite{DV}, if for any neighborhood assignment $\varphi$ for
$X$ there is a closed discrete subset $D$ of $X$ such that
$X=\bigcup_{d\in D}\varphi (d)$.}, then $X$ is compact.
\end{lemma}

\begin{lemma}\label{l9}\cite{AA}
Let $\epsilon$ be the class of all topological spaces $X$ such that the $\pi$-character of $X$ is countable at a dense set of points. Then each Tychonoff space $Y$ which is the union of a finite family of spaces belonging to $\epsilon$ also belongs to $\epsilon$.
\end{lemma}

\begin{theorem}\label{t6}
Assume that $G$ is a non-locally compact rectifiable space, and that $Y$ is the remainder of some Hausdorff compactification of $G$. If the following conditions hold, then $G$ is metrizable.
\begin{enumerate}
\item every $\omega$-bounded\footnote{A space $X$ is said to be {\it
$\omega$-bounded} if the clourse of every countable subset of $X$ is
compact.} subspace of $Y$ is compact;

\item $Y$ is of countable $\pi$-character at a dense set of points.
\end{enumerate}
\end{theorem}

\begin{proof}
By the proof of \cite[Proposition 2.4]{AA}, we can see that $G$ is of countable $\pi$-character. Since $G$ is a rectifiable space, it follows from \cite{G1996} or Lemma~\ref{l4} that $G$ is metrizable.
\end{proof}

\begin{theorem}\label{t7}
Assume that $G$ is a non-locally compact rectifiable space, and that the remainder $bG\setminus G=\bigcup_{i=1}^{n} M_{i}$. If the following conditions hold, then $G$ is metrizable.
\begin{enumerate}
\item each $M_{i}$ is a hereditarily $D$-space;

\item each $M_{i}$ is of countable $\pi$-character at a dense set of points.
\end{enumerate}
\end{theorem}

\begin{proof}
By Lemma~\ref{l8}, it is easy to see that each $\omega$-bounded subspace of the remainder $Y$ is compact. It follows from Lemma~\ref{l9} that the $\pi$-character of $Y$ is countable at a dense set of points. Therefore, $G$ is metrizable by Theorem~\ref{t6}.
\end{proof}

It is well known that any Tychonoff space with a $\sigma$-discrete network is a hereditarily $D$-space \cite{AA1}, and that any Tychonoff spaces with a point-countable base is a hereditarily $D$-space \cite{AA2}. Therefore, we have the following corollary by Theorem~\ref{t7}.

\begin{corollary}
Assume that $G$ is a non-locally compact rectifiable space, and that the remainder $bG\setminus G=\bigcup_{i=1}^{n} M_{i}$. If one of the following conditions holds, then $G$ is metrizable.
\begin{enumerate}
\item each $M_{i}$ is first-countable and has a $\sigma$-discrete network;

\item each $M_{i}$ is metrizable;

\item each $M_{i}$ has a point-countable base.
\end{enumerate}
\end{corollary}

\begin{proposition}\label{p0}\cite{LLL}
Let $G$ be a rectifiable space with point-countable type. If $O$ is an open neighborhood of $e$, then there exists a compact $p, q$ stable subspace $H$ of countable character in $G$ satisfying $H\subset O$.
\end{proposition}

\begin{lemma}\label{l10}\cite{LLL}
Let $G$ be a regular rectifiable space of countable pseudocharacter. Then $G$ has a regular $G_\delta$-diagonal\footnote{A space $X$ is said to have a {\it regular $G_{\delta}$-diagonal}
if the diagonal $\Delta=\{(x, x): x\in X\}$ can be represented as
the intersection of the closures of a countable family of open
neighborhoods of $\Delta$ in $X\times X$.}.
\end{lemma}

\begin{theorem}\label{t8}
Assume that $G$ is a Moscow rectifiable space, and that $Y$ is a remainder of $G$ in some compactification $bG$ of $G$. Then at least one of the following conditions holds:
\begin{enumerate}
\item the space $G$ contains a topological copy of $D^{\omega_{1}}$, where $D=\{0, 1\}$ has the discrete topology;

\item the space $G$ has a regular $G_{\delta}$-diagonal;

\item the compactum $bG$ of $G$ is the $\check{C}$ech-Stone compactification of the space $Y$, and thus $G$ is the $\check{C}$ech-Stone remainder of $Y$.
\end{enumerate}
\end{theorem}

\begin{proof}
Case 1: The space $G$ is locally compact.

Since $G$ is locally compact, $G$ is of point-countable type. Then there exists a compact $p, q$ stable subspace $H$ of countable character in $G$ by Proposition~\ref{p0}.

Subcase 1.1: The space $G$ is metrizable.

Obviously, $G$ has a regular $G_{\delta}$-diagonal.

Subcase 1.2: The space $G$ is non-metrizable.

Obviously, $H$ is non-metrizable, and therefore, it follows from \cite[3.12.12]{E} that $H$ contains a topological copy of $D^{\omega_{1}}$. Thus the space $G$ contains a topological copy of $D^{\omega_{1}}$.

Case 2: The space $G$ is non-locally compact.

Obviously, $G$ is nowhere locally compact since $G$ is a rectifiable space. Then $Y$ is dense in $bG$, that is, $bG$ is a compactification of the space $Y$.

Suppose that the condition (3) dose not hold. Thus there exist closed subsets $A$ and $B$ of $Y$ and a real-valued continuous function $f$ on $Y$ such that $f(A)={0}$ and $f(B)={1}$, while some point $z\in G$ belongs to the intersection of the closures of $A$ and $B$ in $bG$. It follows from the continuity of $f$ that we can find open subsets $U$ and $V$ of $Y$ containing $A$ and $B$, respectively, such that the closures of $U$ and $V$ in $Y$ are disjoint. Then there exists open subsets $U_{1}$ and $V_{1}$ of $bG$ such that $U=U_{1}\cap Y$ and $V=V_{1}\cap Y$.
Let $F$ be the intersection of the closures of $U_{1}$ and $V_{1}$ in $bG$. Obviously, $F$ is compact. Moreover, since $Y$ is dense in $bG$, we have $U$ and $V$ are dense in $U_{1}$ and $V_{1}$, respectively. It follows that $F\subset G$. Put $W_{1}=U_{1}\cap G$ and $W_{2}=V_{1}\cap G$. Obviously, $F$ is the intersection of the closures of $W_{1}$ and $W_{2}$ in $G$. Since $G$ is a Moscow space, the subspace $F$ is the union of closed $G_{\delta}$-subsets of $G$. Then $G$ contains a non-empty compact $G_{\delta}$-subset $P$.

Subcase 2.1: The subspace $P$ is metrizable.

Then each point of $P$ is a $G_{\delta}$-point in $G$, which implies $G$ is of countable pseudocharacter. Then $G$ has a regular $G_{\delta}$-diagonal by Lemma~\ref{l10}.

Subcase 2.2: The subspace $P$ is non-metrizable.

Since $P$ is a compact $G_{\delta}$-subset, the subspace $P$ is dyadic compactum by \cite[Corollary 5]{V1990}. Since $P$ is non-metrizable, it follows from \cite[3.12.12]{E} that the subspace $P$ contains a topological copy of $D^{\omega_{1}}$.
\end{proof}

\begin{corollary}
Let $G$ be a rectifiable space and $Y=bG\setminus G$.
If the following conditions hold, then the compactum $bG$ is the $\check{C}$ech-Stone compactification of the space $Y$, and thus $G$ is the $\check{C}$ech-Stone remainder of $Y$.
\begin{enumerate}
\item the space $G$ has no regular $G_{\delta}$-diagonal;

\item the space $G$ is of countable tightness.
\end{enumerate}
\end{corollary}

\begin{proof}
Obviously, $G$ is nowhere locally compact, since otherwise $G$ would be metrizable by \cite[Theorem 3.4]{G1996}.

Since $G$ is of countable tightness, $G$ is Moscow by \cite[Theorem 5.10]{LF}. Since the tightness of $G$ is countable, it follows from \cite[3.12.12]{E} that $G$ contains no copy of $D^{\omega_{1}}$. Therefore, $G$ is the $\check{C}$ech-Stone remainder of $Y$ by Theorem~\ref{t8}.
\end{proof}

\begin{theorem}
Let $G$ be a rectifiable space. If for each $y\in Y=bG\setminus G$
there exists an open neighborhood $U(y)$ of $y$ such that every
$\omega$-bounded subset of $U(y)$ is compact, then at least one of
the following conditions holds:
\begin{enumerate}
\item $G$ is metrizable;

\item $bG$ can be continuously mapped onto the Tychonoff cube $I^{\omega_{1}}$.
\end{enumerate}
\end{theorem}

\begin{proof}
Case 1: The space $G$ is locally compact.

If $G$ is not metrizable, then $G$ contains a topological copy of
$D^{\omega_{1}}$ by the proof of Theorem~\ref{t8}. Since the space $G$ is normal, the space $G$ can be
continuously mapped onto the Tychonoff cube $I^{\omega_{1}}$

Case 2: The space G is not locally compact.

Obviously, both $G$ and $Y$ are dense in $bG$. Suppose that the
condition (2) doesn't hold. Then, by a theorem of
$\check{S}$apirovski\v{\i} in \cite{SB1}, the set $A$ of all points
$x\in bG$ such that the $\pi$-character of $bG$ at $x$ is countable
is dense in $bG$. Since $G$ is dense in $bG$, it can follow that the
$\pi$-character of $G$ is countable at each point of $A\cap G$.

Subcase 2(a): $A\cap G\neq\emptyset$.

Since $G$ is a rectifiable space, it follows that $G$ is first
countable \cite{G1996}, which implies that $G$ is metrizable.

Subcase 2(b): $A\cap G=\emptyset$.

Obviously, $A\subset Y$. For each $y\in Y$, there exists an open
neighborhood $U(y)$ in $Y$ such that $y\in U(y)$ and every
$\omega$-bounded subset of $U(y)$ is compact. Obviously,  $A\cap
U(y)$ is dense of $U(y)$. Also, it is easy to see that $A\cap U(y)$
is $\omega$-bounded subset for $U(y)$. Therefore, $A\cap U(y)$ is
compact. Then $A\cap U(y)=U(y)$, since $A\cap U(y)$ is dense of
$U(y)$. Hence $Y$ is locally compact, a contradiction.
\end{proof}

{\bf Acknowledgements}. First, I have to thank my Ph.D. advisor Shou Lin for his care, guidance and help in the past few years.

Moreover, I wish to thank
the reviewers for the detailed list of corrections, suggestions to the paper, and all her/his efforts
in order to improve the paper.

\end{document}